\documentclass{amsart}
\usepackage[utf8]{inputenc}
\usepackage{amsmath}
\usepackage{amsthm}
\usepackage{amsfonts}
\usepackage{amssymb}
\usepackage[all]{xy}
\usepackage{indentfirst}
\usepackage[pagebackref=true]{hyperref}

\newtheorem{thm}{Theorem}[section]

\newtheorem{theorem}[thm]{Theorem}

\newtheorem{proposition}[thm]{Proposition}

\newtheorem{lemma}[thm]{Lemma}
\newtheorem*{theorem*}{Theorem}

\theoremstyle{definition}

\newtheorem{example}[thm]{Example}

\newtheorem*{ack}{Acknowledgement}
\newtheorem{remark}[thm]{Remark}

\newcommand{\N}{\mathbb{N}} 
\newcommand{\Z}{\mathbb{Z}} 
\newcommand{\R}{\mathbb{R}}
\usepackage{verbatim} 
\usepackage{color}
\newcommand{\G}{\Gamma}
\newcommand{\g}{\gamma}

\DeclareMathOperator{\Ima}{Im}

\title{Almost finiteness and homology of certain non-free actions}
\author{Eduard Ortega and Eduardo Scarparo}
\address{Department of Mathematical Sciences, NTNU, NO-7491 Trondheim, Norway}
\email{eduard.ortega@ntnu.no, eduardo.scarparo@ntnu.no}
\thanks{This work was carried out during the tenure of an ERCIM ‘Alain Bensoussan’ Fellowship Programme.}

\begin{document}

\begin{abstract}
We show that Cantor minimal $\Z\rtimes\Z_2$-systems and essentially free amenable odometers are almost finite. We also compute the homology groups of Cantor minimal $\Z\rtimes\Z_2$-systems and show that the associated transformation groupoids satisfy the HK conjecture if and only if the action is free.

\end{abstract}

\maketitle

\section{Introduction}

The property of almost finiteness for ample groupoids was introduced by Matui and applied to questions in groupoid homology  in \cite{MR2876963}. Furthermore, in \cite{ara2020strict}, \cite{kerr2017dimension}, \cite{zbMATH07172069} and \cite{suzuki2017almost}, this property was applied to problems in classification of $C^*$-algebras.

Kerr and Szabó in \cite{zbMATH07172069} and Suzuki in \cite{suzuki2017almost} observed that it is a consequence of work of Downarowicz and Zhang \cite{downarowicz2019symbolic} that any free action of a group with subexponential growth on the Cantor set is almost finite. Moreover, free odometers arising from sequences of finite index normal subgroups of an amenable group were shown to be almost finite by Kerr in \cite{kerr2017dimension}.

On the other hand, there exist interesting examples of non-free odometers. For example, in \cite{scarparo2019homology}, certain non-free $\Z\rtimes\Z_2$-odometers were shown to be counterexamples to the \emph{HK conjecture}, which is a conjecture posted by Matui in \cite{MR3552533} that relates the homology groups of an ample groupoid with the K-theory of its reduced $\mathrm{C}^*$-algebra.

In Section \ref{alf}, we investigate almost finiteness of odometers and of Cantor minimal $\Z\rtimes\Z_2$-systems. 

In particular, we prove that an amenable odometer is essentially free if and only if it is almost finite. The proof uses a result by Kar \emph{et al} (\cite{zbMATH06782228}) which shows that, given an amenable essentially free odometers, the acting group admits a Følner sequence consisting of coset representatives.

We also show that any Cantor minimal $\Z\rtimes\Z_2$-system is almost finite. For that, we use a structure result for Cantor minimal $\Z\rtimes\Z_2$-systems from \cite{MR1245825}.

In Section \ref{hom}, we compute the homology groups of Cantor minimal $\Z\rtimes\Z_2$-systems and conclude that these systems satisfy the HK conjecture if and only if the action is free. It should come as no surprise that, in determining the homology groups, we are able to mainly follow the ideas introduced by Bratteli \emph{et al} in \cite{MR1245825} and Thomsen in \cite{MR2586354} in their computation of the K-theory of the associated crossed products since, although the HK conjecture is now known to be false, it has been verified in many cases (see the work of Proietti and Yamashita \cite[Remark 4.7]{proietti2020homology} and the references therein).

\section{Almost finiteness}\label{alf}

In this section, we review some terminology about étale groupoids and verify almost finiteness for certain classes of non-principal groupoids. The reader can find an introduction to étale groupoids in e.g. \cite[Chapter 3]{putnam}.

\subsection{Almost finite groupoids} 

Let $G$ be a Hausdorff étale groupoid with range and source maps denoted by $r$ and $s$. A \emph{bisection} is a subset $S\subset G$ such that $r|_S$ and $s|_S$ are injective maps. 

Let $G':=\{g\in G:r(g)=s(g)\}$ and $G^{(0)}$ be the unit space of $G$. Then $G$ is said to be \emph{principal} if $G'=G^{(0)}$ and \emph{effective} if the interior of $G'$ equals $G^{(0)}$. The \emph{orbit} of a point $x\in G^{(0)}$ is the set $G(x):=r(s^{-1}(x))$. We say that $G$ \emph{minimal} if the orbit of each $x\in G^{(0)}$ is dense in $G^{(0)}$. Also $G$ is said to be  \emph{ample} if its unit space is totally disconnected.

\begin{example}
Let $X$ be a locally compact Hausdorff space and let $\alpha:\Gamma \curvearrowright X$ be an action of a discrete  group $\Gamma$ on $X$. Given $g\in \Gamma$ and $x\in X$ we will denote by $g x:=\alpha(g)(x)$.
As a space, the \emph{transformation groupoid} $G$ associated with $\alpha$ is $G:=\Gamma\times X$ equipped with the product topology.
The product of two elements $(h,y), (g,x)\in G$ is defined if and only if $y=gx$, in which case $(h,gx)(g,x):=(hg,x)$.
Inversion is given by $(g,x)^{-1}:=(g^{-1},gx)$. The unit space $G^{(0)}$ is naturally identified with $X$. Note that $G$ is principal if and only if $\alpha$ is free ($gx=x\Rightarrow g=e$). If $X$ is totally disconnected, then $G$ is ample.
\end{example}

Let $G$ be an ample groupoid with compact unit space. A subgroupoid $K\subset G$ is said to be \emph{elementary} if $K$ is compact-open, principal and $K^{(0)}=G^{(0)}$. The groupoid $G$ is said to be \emph{almost finite} if for any compact subset $C\subset G$ and $\epsilon>0$, there is $K\subset G$ elementary subgroupoid such that, for any $x\in G^{(0)}$, we have that

$$\frac{|CKx\setminus Kx|}{|K(x)|}<\epsilon.$$

For compact groupoids, the following holds:

\begin{proposition}
Let $G$ be an ample almost finite groupoid. If $G$ is compact, then $G$ is principal. 
\end{proposition}
\begin{proof}
Compactness of $G$ implies that there is a finite partition of $G$ into compact-open bisections. Hence, there is $M>0$ such that, for each $x\in G^{(0)}$, we have that $|G(x)|<M$.

If $G$ is not principal, there is $g\in G\setminus G^{(0)}$ such that $r(g)=s(g)$. Then, for any elementary subgroupoid $K\subset G$, we have that $g\notin K$ and

$$\frac{|gKs(g)\setminus Ks(g)|}{|K(s(g))|}>\frac{1}{M}.$$

Therefore, $G$ is not almost finite. 

\end{proof}

Given an étale groupoid $G$ with compact unit space, denote by $M(G)$ the set of $G$-invariant probability measures on $G^{(0)}$. 

\begin{remark} \label{ess}
 It was observed in \cite[Remark 6.6]{MR2876963} that if $G$ is an almost finite groupoid, then $M(G)\neq\emptyset$ and, given $\mu\in M(G)$, we have that 

\begin{equation}\label{essfree}
\mu(r(U\cap(G'\setminus G^{(0)})))=0,\text{ for any $U\subset G$ compact-open bisection.}
\end{equation}

\end{remark}

\subsection{Almost finite actions}
Let us recall the characterization of almost finiteness for transformation groupoids presented in \cite{zbMATH07172069} and \cite{suzuki2017almost}. We will restrict ourselves to actions on totally disconnected spaces.

Let $\G$ be a group acting on a compact Hausdorff totally disconnected space $X$. A \emph{clopen tower} is a pair $(V, S)$ consisting of a clopen subset $V$ of $X$ and a finite subset $S$ of $\G$ such that the sets $sV$ for $s\in S$ are pairwise disjoint. The set $S$ is said to be the \emph{shape} of the tower. A \emph{clopen castle} is a finite collection of clopen towers $\{(V_i, S_i)\}_{i\in I}$ such that the sets $S_iV_i=\{gV_i:g\in S_i\}$ for $i\in I$ are pairwise disjoint.

Given $\epsilon>0$ and $K\subset \G$ finite, we say that a finite set $F\subset\G$ is \emph{$(K,\epsilon)$-invariant} if $|KF\triangle F|<\epsilon|F|$.

Let $G$ be the transformation groupoid associated to the action of $\G$ on $X$. Then $G$ is almost finite if and only if, given $K\subset \G$ finite and $\epsilon>0$, there is a clopen castle which partitions $X$, and whose shapes are $(K,\epsilon)$-invariant (see \cite[Lemma 5.2]{suzuki2017almost} for a proof of this fact). In this case, we say that the \emph{action is almost finite}.

Recall that an action of a group $\G$ on a locally compact Hausdorff space $X$ is said to be \emph{minimal} if the orbit of any $x\in X$ is dense in $X$. Clearly, this is equivalent to the associated transformation groupoid being minimal. It is also equivalent to every open (or closed) $\G$-invariant set being trivial. If the action is minimal, then any $f\in C(X,\Z)$ which is $\G$-invariant must be constant. We will say that a homeomorphism $\varphi$ on $X$ is minimal if the $\Z$-action induced by $\varphi$ is minimal. By a \emph{Cantor minimal $\G$-system}, we mean a minimal action of $\G$ on the Cantor set.

Given $\G$ a group acting on a compact Hausdorff space $X$, we denote by $M_\G(X)$ the set of $\G$-invariant probability measures on $X$. For $g\in\G$, let $\mathrm{Fix}_g\subset X$ be the set of points fixed by $g$.

If $\mu\in M_\G(X)$, we say that the action of $\G$ on $(X,\mu)$ is \emph{essentially free} if, for any $g\in\G\setminus\{e\}$, we have $\mu(\mathrm{Fix}_g)=0$. The action is said to be \emph{topologically free} if the interior of $\mathrm{Fix}_g$ is empty for each $g\in\G\setminus\{e\}$. Notice that topological freeness is equivalent to effectiveness of the associated transformation groupoid. 

Also observe that if $\G\curvearrowright X$ is a minimal action, then any $\mu\in M_\G(X)$ has full support. Therefore, for minimal actions, essential freeness of $\G\curvearrowright(X,\mu)$ implies topological freeness of $\G\curvearrowright X$.

\begin{remark}\label{afifree}
If $G$ is a transformation groupoid associated to an action of a group $\G$ on a compact Hausdorff space $X$ and $\mu\in M(G)\approx M_\G(X)$, then the condition in \eqref{essfree} is equivalent to the action of $\G$ on $(X,\mu)$ being essentially free. It follows from Remark \ref{ess} that if the action is almost finite, then $\G\curvearrowright (X,\mu)$ is essentially free.
\end{remark}

\subsection{Odometers}

Let $\G$ be a group and $(\Gamma_i)_{i\in\N}$ a sequence of finite index subgroups of $\Gamma$ such that, for every $i\in \N$, $\Gamma_i \gneq \Gamma_{i+1}$. 

For each $i\in\mathbb{N}$, let $p_i\colon \G /\Gamma_{i+1}\to\Gamma/\Gamma_{i}$ be the surjection given by 
\begin{equation}
p_i(\gamma\Gamma_{i+1}):=\gamma\Gamma_i,\text{ for } \gamma\in\Gamma.\label{proje}
\end{equation} 
Let $X:=\varprojlim(\G/\G_i,p_i)=\{(x_i)\in\prod\Gamma/\Gamma_i:p_i(x_{i+1})=x_i,\forall i\in\mathbb{N}\}$. Then $X$ is homeomorphic to the Cantor set and $\Gamma$ acts in a minimal way on $X$ by $\gamma(x_i):=(\gamma x_i)$, for $\g\in\G$ and $(x_i)\in X$. This action is called an \emph{odometer}.

Given $j\geq 1$ and $g\G_j\in \G/\G_j$, let $U(j,g\G_j):=\{(x_i)\in X:x_j=g\G_j\}$. Then $\{U(j,g\G_j):j\in\mathbb{N},g\G_j\in\G/\G_j\}$ is a basis for $X$ consisting of compact-open sets.

Notice that $X$ admits a unique $\G$-invariant probability measure. Hence, there is no ambiguity in calling an odometer action essentially free. It is known (see, for example, \cite[(1.6)]{MR2783933}) that the odometer is essentially free if and only if, for each $g\in\G\setminus\{e\}$, we have
$$\lim_i\frac{|\{x\in\G/\G_i:gx=x\}|}{|\G/\G_i|}=0.$$

If the odometer is free, then $\bigcap_i\G_i=\{e\}$, but the converse does not hold in general (see e.g. \cite[Theorem 1]{abert2007non}).

Let us now give a characterization of almost finite odometers. First recall that a Følner sequence $(F_n)$ for $\G$ is a sequence of finite subsets of $\G$ such that for every $g\in\G$ we have that $\frac{|gF_n\Delta F_n|}{|F_n|}\to 0$.  

\begin{theorem}[\cite{zbMATH06782228},\cite{zbMATH00006493}]\label{main}

Let $\G$ be a countable amenable group and $\G\curvearrowright X:=\varprojlim\G/\G_i$ an  odometer. The following conditions are equivalent:
\begin{enumerate}
\item[(i)] The action is essentially free;
\item[(ii)] There is a Følner sequence $(F_n)$ for $\G$ such that each $F_n$ is a complete set of representatives for $\G/\G_n$;
\item[(iii)] The action is almost finite;
\item[(iv)] There is a unique tracial state on $C(X)\rtimes\G$.
\end{enumerate}
 \end{theorem}
\begin{proof}
 The equivalence of (i) and (iv) is a consequence of \cite[Corollary 2.8]{zbMATH00006493}. That (iii) implies (i) is a consequence of Remark \ref{afifree}. The implication from (i) to (ii) is the content of \cite[Theorem 7]{zbMATH06782228}. Let us show then that (ii) implies (iii).

Notice that for each $n\in\N$,

\begin{equation}\label{castelo}
X=\bigsqcup_{\g\in F_n}\g U(n,\G_n).
\end{equation}

By taking $n$ sufficiently big, we can make the shapes of the castle \eqref{castelo} arbitrarily invariant.
\end{proof}

A few remarks are in order about the result above:

\begin{remark}
(i) It follows from Theorem \ref{main} and the results in \cite{ara2020strict} that crossed products associated to essentially free odometers in which the acting group is countable and amenable are classifiable by their Elliott-invariant.

(ii) It is a consequence of \cite[Proposition 4.7]{MR3694599} that if $\G\curvearrowright\varprojlim\G/\G_i$ is an almost finite odometer, then $\G$ is amenable.

(iii) In \cite[Proposition 2.1]{scarparo2019homology}, it was shown that an odometer $\G\curvearrowright X:=\varprojlim\G/\G_i$ is topologically free if and only if, for every $\gamma\in\bigcap\Gamma_i\setminus\{e\}$ and $j\geq 1$, there exists $b\in \Gamma_j$ such that $b^{-1}\gamma b\notin\bigcap\Gamma_i$. 

We do not know whether, for a countable amenable group $\G$, topological freeness of a $\G$-odometer implies essential freeness. If one drops the amenability assumption, there are several counterexamples in the literature (see, for example, \cite[Theorem 1]{abert2007non}).
\end{remark}

\begin{example}\label{ce}

Recall that the \emph{infinite dihedral group} is the semidirect product $\Z\rtimes\Z_2$ associated to the action of $\Z_2$ on $\Z$ by multiplication by $-1$. 

Let $(n_i)$ be a strictly increasing sequence of natural numbers such that $n_i|n_{i+1}$, for every $i\in\mathbb{N}$. Define $\Gamma:=\mathbb{Z}\rtimes\mathbb{Z}_2$ and, for $i\geq 1$, $\Gamma_i:=n_i\mathbb{Z}\rtimes\mathbb{Z}_2$.

By \cite[Lemma 3.2]{scarparo2019homology} and the fact that any element of the form $(n,1)\in\Z\rtimes\Z_2$ is conjugate to either $(0,1)$ or $(1,1)$, we obtain that any $g\in\G\setminus\{e\}$ fixes at most finitely many points in $\varprojlim\G/\G_i$. Hence, this odometer is essentially free.  

Let us describe a Følner sequence for $\Z\rtimes\Z_2$ satisfying condition (ii) in Theorem \ref{main}. Given $m\in\N$, let $F_m\subset\Z\rtimes\Z_2$ be defined by

\begin{align}\label{foelner}
F_m=\begin{cases}
([-\frac{m}{2},-1]\times\{1\})\cup([0,\frac{m}{2})\times\{0\}), & \text{if $m$ is even}\\
([-\frac{m-1}{2},-1]\times\{1\})\cup([0,\frac{m-1}{2}]\times\{0\}), & \text{if $m$ is odd.}\end{cases}
\end{align}
Then $(F_m)_{m\in\N}$ is a Følner sequence for $\Z\rtimes\Z_2$ such that each $F_m$ is a set of representatives for $\frac{\Z\rtimes\Z_2}{(m\Z)\rtimes\Z_2}$.

\end{example}

\subsection{Almost finiteness of Cantor minimal $\Z\rtimes\Z_2$-systems}

Notice that any action $\alpha$ of $\Z\rtimes\Z_2$ on a set $X$ is given by a pair of bijections on $X$ $(\varphi,\sigma)$ such that $\sigma^2=\mathrm{Id}_X$ and $\sigma\varphi\sigma=\varphi^{-1}$, so that $\alpha_{n,i}=\varphi^n\sigma^i$, for $(n,i)\in\Z\rtimes\Z_2$.

There is an isomorphism $\Z\rtimes\Z_2\simeq\Z_2*\Z_2$ which takes the canonical generators $a,b\in\Z_2*\Z_2$ to $(1,1),(0,1)\in\Z\rtimes\Z_2$.

\begin{proposition}\label{fixed}
Let $\alpha:=(\varphi,\sigma)$ be a minimal action of $\Z\rtimes\Z_2$ on the Cantor set $X$. The following holds:

(i) The $\Z$-action induced by $\varphi$ is free and $\alpha$ is topologically free. If $\alpha$ is not free, then either $\sigma$ or $\varphi\sigma$ has at least one fixed point.

(ii) If the $\Z$-action induced by $\varphi$ is not minimal, then there exists a clopen set $Y$ such that $Y\cap\sigma(Y)=\emptyset$, $Y\cup\sigma(Y)=X$, $\varphi(Y)=Y$ and $\varphi|_Y$ is minimal. In particular, $\alpha$ is free.

\end{proposition}

\begin{proof}

(i) Suppose that the $\Z$-action induced by $\varphi$ is not free. Then there is $n\in\Z\setminus\{0\}$ and $x\in X$ such that $\varphi^n(x)=x$, then $\varphi^n\sigma(x)=\sigma\varphi^{-n}(x)=\sigma(x)$. Hence, the $\alpha$-orbit of $x$ is finite, which contradicts the fact that $\alpha$ is minimal. Therefore, the $\Z$-action induced by $\varphi$ is free.

Suppose that $\alpha$ is not topologically free. Then there is a non-empty open set $U\subset X$ and $n\in\Z$ such that $\varphi^n\sigma$ fixes $U$ pointwise. 

 Fix $x\in U$. By minimality of $\alpha$, there is $(m,i)\in\Z\rtimes\Z_2$ such that $\varphi^m\sigma^i(x)\in U$ and $\varphi^m\sigma^i(x)\neq x$. Furthermore, by multiplying $(m,i)$ on the left by $(n,1)$, we can assume that $i=0$.

Since $\varphi^m(x)\in U$, we get that $\varphi^m(x)=\varphi^n\sigma\varphi^m(x)=\varphi^{-m}\varphi^n\sigma(x)=\varphi^{-m}(x),$ hence $\varphi^{2m}(x)= x$. But this contradicts the fact that the $\Z$-action induced by $\varphi$ is free. Therefore, $\alpha$ is topologically free.

If $\alpha$ is not free, then there is $n\in\Z$ and $x\in X$ such that $\varphi^n\sigma(x)=x$. Since any element of the form $(n,1)\in\Z\rtimes\Z_2$ is conjugate to $(0,1)$ or $(1,1)$, we conclude that either $\sigma$ or $\varphi\sigma$ has at least one fixed point.

(ii) This is the content of \cite[Proposition 2.10]{MR3779956} (see also \cite[Lemma 4.28]{MR2586354}).

\end{proof}

The following lemma is a slight modification of \cite[Lemma 1.4]{MR1245825}, and we include the proof for the sake of completeness.

\begin{lemma}\label{tecmain}

Let $(\varphi,\sigma)$ be a minimal action of $\Z\rtimes\Z_2$ on the Cantor set $X$ and $Y\subset X$ a non-empty clopen set such that $\sigma(Y)=Y$. 

Then there is a partition of $Y$ into clopen sets $Y_1,\dots, Y_K$ and positive integers $J_1<J_2\dots<J_K$ such that $\{\varphi^k(Y_i):i=1,\dots,K, k=0,\dots,J_i-1\}$ is a partition of $X$, and $\sigma\varphi^{J_i}(Y_i)=Y_i$ for each $i=1,\dots,K$.

\end{lemma}
\begin{proof}
 Define $\lambda(y):=\min\{n>0:\varphi^n(y)\in Y\}$, for $y\in Y$. From Proposition \ref{fixed}, it follows that the map $\lambda$ is well-defined, in the sense that for each $y\in Y$ there is $n>0$ such that $\varphi^n(y)\in Y$.

It is easy to check that $\lambda$ is continuous. Hence, it has a finite range $\{J_1,\dots,J_K\}$, where $J_1<J_2<\dots<J_K$. Define $Y_i=\lambda^{-1}(J_i)$ for each $i$. Then the sets $\{\varphi^k(Y_i):i=1,\dots,K, k=0,\dots,J_i-1\}$ are pairwise disjoint. From Proposition \ref{fixed}, it follows that the union of these sets is $\varphi$ and $\sigma$-invariant, hence minimality of the action implies that this is a partition of $X$.

Let us show that $\sigma\varphi^{J_i}(Y_i)=Y_i$ for each $i$. Observe first that $\sigma\varphi^{J_i}(Y_i)\subset Y$, since $\varphi^{J_i}(Y_i)\subset Y$ and $Y$ is $\sigma$-invariant. 

Take $y\in\sigma\varphi^{J_i}(Y_i)$. Then there is $x\in Y_i$ such that $y=\sigma\varphi^{J_i}(x)$. In particular, $\varphi^{J_i}(y)=\sigma(x)\in Y$. Therefore, $\lambda(y)\leq J_i$.

If $i=1$, we must then have $\lambda(y)=J_1$. Therefore, $\sigma\varphi^{J_1}(Y_1)\subset Y_1$. Since $\sigma\varphi^{J_1}$ is involutive, we conclude that $\sigma\varphi^{J_1}(Y_1)= Y_1$. Arguing the same way for $i=2,\dots,K$, the result follows.

\end{proof}

\begin{theorem}
Any minimal action of $\Z\rtimes\Z_2$ on the Cantor set is almost finite.
\end{theorem}
\begin{proof}
Let $\alpha=(\varphi,\sigma)$ be a minimal action of $\Z\rtimes\Z_2$ on the Cantor set $X$. 
It follows from the fact that $\alpha$ is topologically free (Proposition \ref{fixed}) and the Baire Category Theorem that there is $y\in X$ such that, if $(n,i),(m,j)\in\Z\rtimes\Z_2$ are distinct, then $\varphi^n\sigma^i(x)\neq\varphi^m\sigma^j(x)$. 

Let $Z$ be a clopen neighborhood of $y$ and $Y:=Z\cup\sigma(Z)$. Then $\sigma(Y)=Y$ and, given $N\in\N$, if we take $Z$ sufficiently small, we can assume that the sets $\{Y,\varphi(Y),\dots,\varphi^{N-1}(Y)\}$ are disjoint. Then, by Lemma \ref{tecmain}, we can partition $Y$ into clopen sets $Y_1,\dots,Y_K$ such that 
$$X=\sqcup_{i=1}^K\sqcup_{k=0}^{J_i-1}\varphi^{k}(Y_i)$$
and $\sigma\varphi^{J_i}(Y_i)=Y_i$ for each $i=1,\dots,K$. In particular, $\varphi^{-l}\sigma(Y_i)=\varphi^{J_i-l}(Y_i)$ for any $l$.

Consider the Følner sequence $(F_m)$ introduced in \eqref{foelner}, and notice that 

\begin{equation}\label{castle}
X=\bigsqcup_{i=1}^K\bigsqcup_{g\in F_{J_i}}\alpha_g(Y_i).
\end{equation}

Furthermore, each $J_i\geq N$. By taking $N$ sufficiently big, we can make the shapes of the castle \eqref{castle} arbitrarily invariant.

\end{proof}

\section{Homology of Cantor minimal $\Z\rtimes\Z_2$-systems}\label{hom}

In this section, we compute the homology groups of Cantor minimal $\Z\rtimes\Z_2$-systems.

Given a group $\G$ and a $\G$-module $M$, we denote by $M_\G$ the quotient of $M$ by the subgroup generated by elements of the form $m-mg$, for $m\in M$ and $g\in\G$. Recall that $M_\G$ is canonically isomorphic to $H_0(\G,M)$. If $i\colon\Lambda\to\G$ is an embedding, we denote the canonical map $i_*\colon H_*(\Lambda,M)\to H_*(\G,M)$ by $\mathrm{cor}$.

We will use the following result about the homology of free products, whose proof can be found in \cite[Theorem 2.3]{MR240177}.
\begin{theorem}\label{freep}

Let $\G_1$ and $\G_2$ be groups and $M$ a $(\G_1*\G_2)$-module. Then, for $n\geq 2$, 
$$H_n(\G_1,M)\oplus H_n(\G_2,M)\simeq H_n(\G_1*\G_2,M)$$ and there is an exact sequence\begin{align*}
0&\longrightarrow H_1(\G_1,M)\oplus H_1(\G_2,M)\stackrel{(\mathrm{cor},\mathrm{cor})}{\longrightarrow} H_1(\G_1*\G_2,M)\\
&\longrightarrow M\stackrel{(\mathrm{cor},-\mathrm{cor})}{\longrightarrow} M_{\G_1}\oplus M_{\G_2}\stackrel{(\mathrm{cor},\mathrm{cor})}{\longrightarrow} M_{\G_1*\G_2}\longrightarrow 0
\end{align*}

\end{theorem}

Given an action of a group $\Gamma$ on a topological space $X$, then $C(X,\Z)$ has a structure of $\Gamma$-module given by $fa(x):=f(ax)$ for every $f\in C(X,\Z)$ and $a\in\G$.

\begin{lemma}\label{se}
Let $a$ and $b$ be involutive homeomorphisms on the Cantor set $X$ and suppose that the $\Z$-action induced by $ab$ is minimal. Denote by $A$ and $B$ the abelian group $C(X,\Z)$ endowed with the $\Z_2$-action given by $a$ and $b$, respectively.

Then $(\mathrm{cor},\mathrm{cor})\colon H_1(\Z_2,A)\oplus H_1(\Z_2,B)\to H_1(\Z_2*\Z_2,C(X,\Z))$ is an isomorphism and the following sequence is exact:

$$0\longrightarrow C(X,\Z)\stackrel{(\mathrm{cor},-\mathrm{cor})}{\longrightarrow}A_{\Z_2}\oplus B_{\Z_2}\stackrel{(\mathrm{cor},\mathrm{cor})}{\longrightarrow} C(X,\Z)_{\Z_2*\Z_2}\longrightarrow 0$$

\end{lemma}

\begin{proof}
By Theorem \ref{freep}, we only need to show that $$(\mathrm{cor},-\mathrm{cor})\colon C(X,\Z)\to A_{\Z_2}\oplus B_{\Z_2}$$ is injective.

Take $f\in C(X,\Z)$ such that $(\mathrm{cor},-\mathrm{cor})(f)=0$. This implies that there exist $g,h\in C(X,\Z)$ such that $f=g-g a=h-h b$. In particular, $f a=f b=-f$. Therefore, $f ab=f$. Since the $\Z$-action induced by $ab$ is minimal, we conclude that $f$ is constant. Finally, as $fa=-f$, we must have $f=0$.

\end{proof}

In order to apply Theorem \ref{freep} and Lemma \ref{se} to Cantor minimal $\Z_2*\Z_2$-systems, we need to compute homology groups of the form $H_*(\Z_2,C(X,\Z))$.

\begin{lemma}\label{impar}

Let $a$ be an involutive homeomorphism on a compact, Hausdorff, totally disconnected space $X$. Then, for $k\geq 0$, we have $H_{2k+1}(\Z_2,C(X,\Z))\simeq C(\mathrm{Fix}_a,\Z_2)$. 

\end{lemma}

\begin{proof}
By \cite[Theorem 6.2.2]{MR1269324}, we have that 
$$H_{2k+1}(\Z_2,C(X,\Z))=\frac{\{f\in C(X,\Z):f=f a\}}{\{f+f a:f\in C(X,\Z)\}}.$$

Let $E\colon\frac{\{f\in C(X,\Z):f=f a\}}{\{f+f a:f\in C(X,\Z)\}}\to C(\mathrm{Fix}_a,\Z_2)$ be the map  given by restriction to $\mathrm{Fix}_a$. Clearly, this is a well-defined homomorphism, and we will show that it is bijective.

Given $F\subset \mathrm{Fix}_a$ clopen, take for each $x\in F$ a clopen set $U_x\subset X$ such that $x\in U_x$, $U_x\cap \mathrm{Fix}_a\subset F$ and $a(U_x)=U_x$. Then there exist $x_1,\dots,x_n\in F$ such that $F=\bigcup (U_{x_i}\cap \mathrm{Fix}_a)$. Let $U:= \bigcup U_{x_i}$. Notice that $a(U)=U$. Hence, $E([1_U])=1_F$. Therefore, $E$ is surjective.

Let us show now injectivity of $E$. Take $f\in C(X,\Z)$ such that $f=f a$ and $E([f])=0$, and we will show that $[f]=0$. We claim that we can assume that $f|_{\mathrm{Fix}_a}=0$. Indeed, let $U_1,\dots,U_n$ be $a$-invariant clopen subsets of $X$ whose union cover $\mathrm{Fix}_{a}$, and such that $f$ is constant on each $U_i$. By taking differences, we can assume that these sets are disjoint. By summing $f$ with functions of the form $2m_i1_{U_i}$, we get our claim.

Assume then that $f|_{\mathrm{Fix}_a}=0$. We have $f=\sum_{q\in\Z\setminus\{0\}}q1_{f^{-1}(q)}$. Since the support of $f$ does not intersect $\mathrm{Fix}_a$, we can, for each $q\in\Z\setminus\{0\}$, partition $f^{-1}(q)$ as $f^{-1}(q)
=A_q\sqcup a(A_q)$, for some $A_q$ clopen. Hence, $[f]=0$.
\end{proof}

\begin{lemma}\label{h2}
Let  $a$ be an involutive homeomorphism on a compact, Hausdorff, totally disconnected space $X$ and $G_a:=\{f\in C(X,\Z):\text{$f a = f$ and $f(\mathrm{Fix}_a)\subset 2\Z$}\}$. Then 
\begin{align*}
\psi\colon C(X,\Z)_{\Z_2}&\to G_a\\
[f]&\mapsto f+f a
\end{align*}
is an isomorphism and, for $k\geq 1$, we have $H_{2k}(\Z_2,C(X,\Z))= 0$.

\end{lemma}
\begin{proof}
Clearly, $\psi$ is a well-defined homomorphism, and we will show that it is bijective.

Given $f\in C(X,\Z)$, suppose $f+f a=0$. Then $f|_{\mathrm{Fix}_a}=0$. Take $A$ clopen neighborhood of $\mathrm{Fix}_a$ such that $a(A)=A$ and $f$ vanishes on $A$.

Then we can partition $A^c$ as $A^c=B\sqcup a(B)$ for some clopen set $B$. Let $g:=f1_B$. Then $f=g-g a$. Hence, $\psi$ is injective.

Let us now show surjectivity of $\psi$. Take $f\in C(X,\Z)$ such that $f a=f$ and $f(\mathrm{Fix}_a)\subset 2\Z$, and let us show that $f\in\Ima\psi$.

Let $U_1,\dots,U_n$ be $a$-invariant clopen subsets of $X$ whose union cover $\mathrm{Fix}_{a}$, and such that $f$ is constant on each $U_i$. By taking differences, we can assume that these sets are disjoint. By summing $f$ with functions of the form $2m_i1_{U_i}$, we may assume that $f|_{\mathrm{Fix}_a}=0$. We have $f=\sum_{q\in\Z\setminus\{0\}}q1_{f^{-1}(q)}$. Since the support of $f$ does not intersect $\mathrm{Fix}_a$, we can, for each $q\in\Z\setminus\{0\}$, partition $f^{-1}(q)$ as $f^{-1}(q)
=A_q\sqcup a(A_q)$, for some $A_q$ clopen. Let $g:=\sum_{q\in\Z\setminus\{0\}}qA_q$. Then $\psi([g])=f$.

Finally, notice that by \cite[Theorem 6.2.2]{MR1269324}, we have that, for $k\geq 1$, 
\begin{equation}\label{chato}
H_{2k}(\Z_2,C(X,\Z))=\frac{\{f\in C(X,\Z):f+fa=0\}}{\{f-fa:f\in C(X,\Z)\}}.
\end{equation}
Since the right-hand side of \eqref{chato} is equal to $\ker\psi$, the result follows.

\end{proof}

Given a finite index subgroup $\Lambda$ of a group $\G$, and $M$ a $\G$-module, there is a homomorphism $\mathrm{tr}\colon M_\G\to M_\Lambda$ given by $\mathrm{tr}([m])=[\sum g_im]$, where $\{g_1,\dots,g_{[\G:\Lambda]}\}$ is a complete set of right coset representatives (this is the so called \emph{transfer map}, and it does not depened on the choice of representatives).

\begin{theorem}\label{transfer}
Let $\alpha:=(\varphi,\sigma)$ be an action of $\Z\rtimes\Z_2$ on the Cantor set $X$ such that the restricted $\Z$-action is minimal. Let 

\begin{align*}
\mathrm{tr}\colon H_0(\Z\rtimes\Z_2,C(X,\Z))&\to (1+\sigma_*)H_0(\Z,C(X,\Z))\\
 [f]&\mapsto[f+f\sigma].
\end{align*}

If $\alpha$ is not free, then $\mathrm{tr}$ is an isomorphism. If $\alpha$ is free, then $\ker \mathrm{tr}\simeq \Z_2$ and it is generated by $[1_K]-[1_L]$ where $K$ and $L$ are clopen sets such that $X=K\sqcup\sigma(K)=L\sqcup\varphi\sigma(L)$.

\end{theorem}
\begin{proof}

Let $f\in \ker\mathrm{tr}$. In this case, there is $h\in C(X,\Z)$ such that 
\begin{equation}\label{equ}
f+f\sigma=h-h\varphi.
\end{equation}
This implies that $h-h\varphi=(h-h\varphi)\sigma=h\sigma-h\varphi\sigma$.

Therefore, $h+h\varphi\sigma=h\sigma+h\varphi$. Composing the right-hand side of this equation with $\varphi^{-1}$, we obtain that $h+h\varphi\sigma$ is $\varphi$-invariant. Since, $\varphi$ is minimal, we conclude that there is an integer $z$ such that 
\begin{align}
\label{aga}h+h\varphi\sigma&=-z
\end{align}

Furthermore, from \eqref{equ}, we get that $f+f\sigma-h=-h\varphi=-h\varphi-h\sigma+h\sigma=z+h\sigma.$ Therefore,

\begin{align}
\label{efe}f+f\sigma-h-h\sigma&=z.
\end{align} 

Let $G_\sigma$ and $G_{\varphi\sigma}$ be as in Lemma \ref{h2}. By Lemmas \ref{se} and \ref{h2}, there is an isomorphism $$\psi\colon C(X,\Z)_{\Z\rtimes\Z_2}\to\frac{G_\sigma\oplus G_{\varphi\sigma}}{\{(g+g\sigma,-g-g\varphi\sigma):g\in C(X,\Z)\}}$$ such that, for $f_1,f_2\in C(X,\Z)$, we have $\psi([f_1+f_2])=[(f_1+f_1\sigma,f_2+f_2\varphi\sigma)]$. 

From \eqref{aga} and \eqref{efe} we get that $\psi([f])=\psi([f+0])=[(f+f\sigma,0)]=[(z,-z)]$.

If $z$ is even, then clearly $[(z,-z)]=0$. Besides, if $\alpha$ is not free, then it follows from \eqref{aga}, \eqref{efe} and Proposition \ref{fixed} that $z$ is even. 

Now suppose $\alpha$ is free and take $K$ and $L$ clopen sets as in the statement. Clearly, $[1_K-1_L]\in\ker\mathrm{tr}$ and $\psi([1_K-1_L])=[(1,-1)]$. Finally, notice that $[(1,-1)]\neq 0$. Indeed, if there is $g\in C(X,\Z)$ such that $1=g+g\sigma=g+g\varphi\sigma$, then $g\sigma=g\varphi\sigma$, which implies that $g$ is $\varphi$-invariant, hence constant. But this contradicts the fact that $1=g+g\sigma$.

\end{proof}

The next result is essentially a summary of what we have obtained so far.

\begin{theorem}\label{comp}
Let $\alpha:=(\varphi,\sigma)$ be a minimal action of $\Z\rtimes\Z_2$ on the Cantor set $X$.
\begin{enumerate}
\item[(i)] If $\varphi$ is not minimal, then there exists $Y\subset X$ clopen and $\varphi$-invariant such that $X=Y\sqcup\sigma(Y)$ and $\varphi|_Y$ is minimal. Furthermore,
\begin{align*}
H_0(\Z\rtimes\Z_2,C(X,\Z))&\simeq H_0(\Z,C(Y,\Z)),\\
H_1(\Z\rtimes\Z_2,C(X,\Z))&\simeq\Z,\\
H_n(\Z\rtimes\Z_2,C(X,\Z))&=0,\text{ for $n\geq 2$};
\end{align*}
\item[(ii)] If $\varphi$ is minimal and $\alpha$ is free, then
\begin{align*}
H_0(\Z\rtimes\Z_2,C(X,\Z))&\simeq \Z_2\oplus (1+\sigma_*)H_0(\Z,C(X,\Z)),\\
H_n(\Z\rtimes\Z_2,C(X,\Z))&=0,\text{ for $n\geq 1$};
\end{align*}
\item[(iii)] If $\alpha$ is not free, then
\begin{align*}
H_0(\Z\rtimes\Z_2,C(X,\Z))&\simeq (1+\sigma_*)H_0(\Z,C(X,\Z)),\\
H_{2n-1}(\Z\rtimes\Z_2,C(X,\Z))&\simeq C(\mathrm{Fix}_\sigma\sqcup \mathrm{Fix}_{\varphi\sigma},\Z_2),\text{ for $n\geq 1$},\\
H_{2n}(\Z\rtimes\Z_2,C(X,\Z))&=0,\text{ for $n\geq 1$}.
\end{align*}

\end{enumerate}
\end{theorem}
\begin{proof}

(i) The existence of $Y$ is the content of Proposition \ref{fixed}. Notice that 
$$C(X,\Z)\simeq\mathrm{Ind}_\Z^{\Z\rtimes\Z_2}C(Y,\Z).$$ 
Shapiro's Lemma \cite[Proposition 6.2]{MR1324339} implies then that $$H_*(\Z\rtimes\Z_2,C(X,\Z))\simeq H_*(\Z,C(Y,\Z)).$$ Finally, the homology groups of a Cantor minimal $\Z$-system are easy to compute.

(ii) Let us begin by showing that 
\begin{equation}\label{seila}
H_0(\Z\rtimes\Z_2,C(X,\Z))\simeq \Z_2\oplus (1+\sigma_*)H_0(\Z,C(X,\Z)).
\end{equation}

Consider the map $\mathrm{tr}\colon H_0(\Z\rtimes\Z_2,C(X,\Z))\to (1+\sigma_*)H_0(\Z,C(X,\Z))$
from Theorem \ref{transfer}. Since $H_0(\Z,C(X,\Z))$ is torsion-free, it follows that its subgroup $(1+\sigma_*)H_0(\Z,C(X,\Z))$ is torsion free as well. Since $\mathrm{tr}$ is surjective, we conclude from \cite[Section 23 (H)]{zbMATH03162956} that $\ker\mathrm{tr}$ is a pure subgroup of $ H_0(\Z\rtimes\Z_2,C(X,\Z))$. By Theorem \ref{transfer}, we have that $\ker\mathrm{tr}\simeq\Z_2$. Therefore, \cite[Theorem 24.1]{zbMATH03162956} implies that $\ker\mathrm{tr}$ is a direct summand of $H_0(\Z\rtimes\Z_2,C(X,\Z))$. This conludes the proof of \eqref{seila}.

The remaining computations of the homology groups in cases (ii) and (iii) are a consequence of Theorem \ref{freep} and Lemmas \ref{se}, \ref{impar} and \ref{h2}. In case (iii), since the $\Z$-action induced by $\varphi$ is free, notice that $\mathrm{Fix}_\sigma$ is disjoint from $\mathrm{Fix}_{\varphi\sigma}$.

\end{proof}

Let us now apply Theorem \ref{comp} to an example which had its K-theory computed in \cite[Corollary 4.4]{MR1245825}.

\begin{example}
Fix $\theta\in(0,1)$ an irrational number. Let $\tilde{X}$ be the set obtained from $\R$ by replacing each $t\in\Z+\theta\Z$ by two elements $\{t^-, t^+\}$, and endow $\tilde{X}$ with the order topology. Notice that there is an action $\Z\stackrel{\alpha}{\curvearrowright}\tilde{X}$ by translations ($\alpha_n(x)=n+x$). We let $X:=\frac{\tilde{X}}{\Z}$. Then $X$ is homeomorphic to the Cantor set and there is a minimal homeomorphism $R_\theta\colon X\to X$ given by $R_{\theta}(x)=x+\theta$.
 
Furthermore, there is an involutive homeomorphism $\sigma$ on $X$ given by $\sigma(x)=-x$ and $\sigma(t^{\pm})=t^{\mp}$. Notice that the only fixed point of $\sigma$ is $\frac{1}{2}$, and the only fixed points of $R_\theta\circ\sigma$ are $\frac{1+\theta}{2}$ and $\frac{\theta}{2}$. 

It was shown in \cite[Theorem 2.1]{MR978619} that $H_0(\Z,C(X,\Z))\simeq \Z^2$, and the generators are $[1_{[0^+,\theta^+)]}]$ and $[1_{[\theta^+,1^+)}]$. Observe that $\sigma_*$ acts trivially on these two elements.

From these observations and Theorem \ref{comp}, it follows that, for $k\geq 1$,
\begin{align*}
H_0(\Z\rtimes\Z_2,C(X,\Z))&\simeq \Z^2\\
H_{2k-1}(\Z\rtimes\Z_2,C(X,\Z))&\simeq\Z_2\times\Z_2\times\Z_2\\
H_{2k}(\Z\rtimes\Z_2,C(X,\Z))&=0.
\end{align*}

\end{example}

\begin{remark}
In \cite{MR3552533}, Matui conjectured that if $G$ is an ample, effective, second countable, minimal groupoid with unit space homemorphic to the Cantor set, then $K_*(C^*_r(G))\simeq\bigoplus_{n=0}^\infty H_{2n+*}(G)$ for $*=0,1$ (\emph{HK conjecture}). In \cite{scarparo2019homology}, the (transformation groupoids of the) $\Z\rtimes\Z_2$-odometers from Example \ref{ce} were shown to be counterexamples to the HK conjecture.

Let $\alpha$ be a minimal action of $\Z\rtimes\Z_2$ on the Cantor set $X$. If $\alpha$ is not free, then the associated crossed product $C(X)\rtimes\Z\rtimes\Z_2$ is AF \cite[Theorem 3.5]{MR1245825}, hence $K_1(C(X)\rtimes\Z\rtimes\Z_2)=0$. On the other hand, it follows from Proposition \ref{fixed} and Theorem \ref{comp} that $H_{2n+1}(\Z\rtimes\Z_2,C(X,\Z))\neq 0$. Therefore, $\alpha$ is a counterexample to the HK conjecture.

If $\alpha$ is free, it follows from \cite[Theorems 4.30 and 4.42]{MR2586354} and Theorem \ref{comp} that $\alpha$ satisfies the HK conjecture. Alternatively, this also follows from Theorem \ref{comp} and \cite[Remark 4.7]{proietti2020homology}.

\end{remark}
\begin{remark}

Given an ample groupoid $G$ with compact unit space, the \emph{topological full group} of $G$, denoted by $[[G]]$ is the the group of compact-open bisections $U$ such that $r(U)=s(U)=G^{(0)}$.

Given an effective, minimal, second countable groupoid $G$ with compact unit space homeomorphic to the Cantor set, Matui conjectured in \cite{MR3552533} that the index map $I\colon[[G]]_\mathrm{ab}\to H_1(G)$ is surjective and that its kernel is a quotient of $H_0(G)\otimes\Z_2$ under a certain canonical map (\emph{AH conjecture}). Assuming that $G$ is also minimal and almost finite, Matui proved that the index map is surjective and if $G$ is also principal, then it satisfies AH conjecture.

We have not been able to verify whether the AH conjecture holds for non-free Cantor minimal $\Z\rtimes\Z_2$-systems in general, but note that in \cite{scarparo2019homology} it was shown that the $\Z\rtimes\Z_2$-odometers from Example \ref{ce} satisfy it.

\end{remark}
\begin{ack}
The authors would like to thank the referee for a careful reading of the paper and the helpful comments.

\end{ack}
\bibliographystyle{acm}
\bibliography{bibliografia}

\begin{thebibliography}{10}

\bibitem{abert2007non}
{\sc Ab{\'e}rt, M., and Elek, G.}
\newblock Non-abelian free groups admit non-essentially free actions on rooted
  trees.
\newblock {\em arXiv preprint arXiv:0707.0970\/} (2007).

\bibitem{ara2020strict}
{\sc Ara, P., B{\"o}nicke, C., Bosa, J., and Li, K.}
\newblock Strict comparison for {$C\sp*$}-algebras arising from almost finite
  groupoids.
\newblock {\em arXiv preprint arXiv:2002.12221\/} (2020).

\bibitem{MR1245825}
{\sc Bratteli, O., Evans, D.~E., and Kishimoto, A.}
\newblock Crossed products of totally disconnected spaces by {$Z_2\ast Z_2$}.
\newblock {\em Ergodic Theory Dynam. Systems 13}, 3 (1993), 445--484.

\bibitem{MR1324339}
{\sc Brown, K.~S.}
\newblock {\em Cohomology of groups}, vol.~87 of {\em Graduate Texts in
  Mathematics}.
\newblock Springer-Verlag, New York, 1994.
\newblock Corrected reprint of the 1982 original.

\bibitem{downarowicz2019symbolic}
{\sc Downarowicz, T., and Zhang, G.}
\newblock Symbolic extensions of amenable group actions and the comparison
  property.
\newblock {\em arXiv preprint arXiv:1901.01457\/} (2019).

\bibitem{zbMATH03162956}
{\sc {Fuchs}, L.}
\newblock {Abelian groups.}
\newblock {Oxford-London-New York-Paris: Pergamon Press. 367 p. (1960).}, 1960.

\bibitem{MR2783933}
{\sc Ioana, A.}
\newblock Cocycle superrigidity for profinite actions of property ({T}) groups.
\newblock {\em Duke Math. J. 157}, 2 (2011), 337--367.

\bibitem{zbMATH06782228}
{\sc {Kar}, A., {Kropholler}, P., and {Nikolov}, N.}
\newblock {On growth of homology torsion in amenable groups.}
\newblock {\em {Math. Proc. Camb. Philos. Soc.} 162}, 2 (2017), 337--351.

\bibitem{zbMATH00006493}
{\sc {Kawamura}, S., {Takemoto}, H., and {Tomiyama}, J.}
\newblock {State extensions in transformation group \(C^*\)-algebras.}
\newblock {\em {Acta Sci. Math.} 54}, 1-2 (1990), 191--200.

\bibitem{kerr2017dimension}
{\sc Kerr, D.}
\newblock Dimension, comparison, and almost finiteness.
\newblock {\em arXiv preprint arXiv:1710.00393\/} (2017).

\bibitem{zbMATH07172069}
{\sc {Kerr}, D., and {Szab\'o}, G.}
\newblock {Almost finiteness and the small boundary property.}
\newblock {\em {Commun. Math. Phys.} 374}, 1 (2020), 1--31.

\bibitem{MR3694599}
{\sc Li, X.}
\newblock Partial transformation groupoids attached to graphs and semigroups.
\newblock {\em Int. Math. Res. Not. IMRN}, 17 (2017), 5233--5259.

\bibitem{MR2876963}
{\sc Matui, H.}
\newblock Homology and topological full groups of \'etale groupoids on totally
  disconnected spaces.
\newblock {\em Proc. Lond. Math. Soc. (3) 104}, 1 (2012), 27--56.

\bibitem{MR3552533}
{\sc Matui, H.}
\newblock Étale groupoids arising from products of shifts of finite type.
\newblock {\em Adv. Math. 303\/} (2016), 502--548.

\bibitem{MR3779956}
{\sc Nekrashevych, V.}
\newblock Palindromic subshifts and simple periodic groups of intermediate
  growth.
\newblock {\em Ann. of Math. (2) 187}, 3 (2018), 667--719.

\bibitem{proietti2020homology}
{\sc Proietti, V., and Yamashita, M.}
\newblock Homology and k-theory of torsion-free ample groupoids and smale
  spaces.
\newblock {\em arXiv preprint arXiv:2006.08028\/} (2020).

\bibitem{MR978619}
{\sc Putnam, I.~F.}
\newblock The {$C^*$}-algebras associated with minimal homeomorphisms of the
  {C}antor set.
\newblock {\em Pacific J. Math. 136}, 2 (1989), 329--353.

\bibitem{putnam}
{\sc Putnam, I.~F.}
\newblock Lecture notes on {$C^*$}-algebras.
\newblock Available at www.math.uvic.ca/faculty/putnam/ln/C*-algebras.pdf,
  2019.

\bibitem{scarparo2019homology}
{\sc Scarparo, E.}
\newblock Homology of odometers.
\newblock {\em Ergodic Theory and Dynamical Systems\/} (2019), 1--11.

\bibitem{suzuki2017almost}
{\sc Suzuki, Y.}
\newblock Almost finiteness for general {\'e}tale groupoids and its
  applications to stable rank of crossed products.
\newblock {\em arXiv preprint arXiv:1702.04875\/} (2017).

\bibitem{MR240177}
{\sc Swan, R.~G.}
\newblock Groups of cohomological dimension one.
\newblock {\em J. Algebra 12\/} (1969), 585--610.

\bibitem{MR2586354}
{\sc Thomsen, K.}
\newblock The homoclinic and heteroclinic {$C^*$}-algebras of a generalized
  one-dimensional solenoid.
\newblock {\em Ergodic Theory Dynam. Systems 30}, 1 (2010), 263--308.

\bibitem{MR1269324}
{\sc Weibel, C.~A.}
\newblock {\em An introduction to homological algebra}, vol.~38 of {\em
  Cambridge Studies in Advanced Mathematics}.
\newblock Cambridge University Press, Cambridge, 1994.

\end{thebibliography}
\end{document}